\documentclass[12pt,letterpaper]{amsart}
\usepackage{mathrsfs,amsmath,amssymb,amsthm,amsfonts}
\usepackage{setspace,color,geometry,mathpazo}
\newtheorem{theorem}{Theorem}
\newtheorem*{theorem*}{Theorem}
\newtheorem{proposition}{Proposition}
\newtheorem{lemma}{Lemma}
\newtheorem{corollary}{Corollary}
\theoremstyle{remark}
\newtheorem{remark}{Remark}

\newtheorem{example}{Example}

\newcommand{\C}{\mathbb{C}}
\newcommand{\re}{\mathbb{R}}
\newcommand{\Lev}{\mathcal{L}}
\newcommand{\D}{\Omega}
\newcommand{\Dc}{\overline{\Omega}}
\newcommand{\sbs}{\subset}
\newcommand{\supp}{\text{Supp}}
\newcommand{\J}{{\mathcal J}}
\newcommand{\ovr}{\overline}
\newcommand{\real}{\text{Re}}

\newcommand{\zb}{\overline{z}}
\onehalfspace %spacing

\title{Continuity of plurisubharmonic envelopes in $\C^2$}

\author{N\.{i}hat G\"{o}khan G\"{o}\u{g}\"{u}\c{s}}
\address[N\.{i}hat G\"{o}khan G\"{o}\u{g}\"{u}\c{s}]{Sabanc\i{} University,
Orhanli Tuzla, 34956, Istanbul, Turkey}
\email{nggogus@sabanciuniv.edu}

\author{S\"{o}nmez \c{S}ahuto\u{g}lu}
\address[S\"{o}nmez \c{S}ahuto\u{g}lu]{University of Toledo, Department of
Mathematics \& Statistics, Toledo, OH 43606, USA}
\email{sonmez.sahutoglu@utoledo.edu}

\thanks{The first author is supported by the Scientific and Technological Research Council
of Turkey related to a grant project called "Jensen measures in complex analysis and
c-regularity" with project number 110T223.}

\date{\today}
\begin{document}

\begin{abstract}
We show that in $\C^2$ if the set of strongly regular points are
closed in the boundary of a smooth bounded pseudoconvex domain, then
the domain is $c$-regular, that is, the plurisubharmonic upper
envelopes of functions  continuous up to the boundary are continuous
on the closure of the domain. Using this result we prove that smooth
bounded pseudoconvex Reinhardt domains in $\C^{2}$ are $c$-regular.
\end{abstract}

\maketitle
%%%%%%%%%%%%%%%%%%%%%%%%%%%%%%%%%%%%%%%%
%%%%%%%%%%%%%%%%%%%%%%%%%%%%%%%%%%%%%%%%

\section{Introduction}    %) A SECTION HEADING
Let $\D$ be a domain in $\C^n$ and $\partial \D$ denote the boundary
of $\D.$ For any upper bounded function $u$ on $\D$ we define the
upper regularization $u^*$ of $u$ on $\Dc$ as the function
\[u^*(z)=\limsup _{ w\in\D,w\to z}u(w)\]
for any $z\in\Dc$. Given $z\in \ovr \D$, we denote by
${\J}_z=\J_z(\D)$ the family of all positive Borel measures $\mu\in
C^*(\ovr \D)$ such that $u^*(z)\leq \int u^*\,d\mu$ for every $u$ in
the set $PSH^b(\D)$ of all upper bounded plurisubharmonic functions
on  $\D$. Such measures are called Jensen measures centered at $z.$
For example, the Dirac measure $\delta_z\in\J_z$ for every
$z\in\ovr\D$. For a nontrivial example, let $z\in\D$ and consider a
map $f:\ovr{\mathbb{D}}\to\D$ which is holomorphic in an open
neighborhood of the closure of the unit disk $\ovr{\mathbb{D}}$ so
that $f(0)=z$. Define
\[\mu_f(\varphi)=\frac{1}{2\pi}\int_0^{2\pi}\varphi\circ f(e^{i\theta})\,d\theta\]
for any function $\varphi\in C(\ovr\D)$. Then $\mu_f$ defines a
measure on $\D$ which belongs to $\J_z$. Let us denote the class of
all measures of the form $\mu_f$, where $f$ is a mapping as above,
by ${\mathcal{H}}_z$. It was proved in \cite{BuSchachermayer92} that
the weak-$*$ closure $\ovr{\mathcal{H}}_z$ in $C^*(\Dc)$ of
${\mathcal{H}}_z$ coincides with $\J_z$ when $z\in\D$. Moreover, by
\cite{Poletsky91} and \cite{Poletsky93} if $\varphi$ is an upper
semicontinuous function on $\D$, then the function
\[I\varphi(z)=\inf \left\{\int \varphi\,d\mu :\, \mu\in{\mathcal{H}}_z\right\}\]
is plurisubharmonic and equal to the plurisubharmonic upper envelope
$E\varphi$ of $\varphi$ on $\D$, where
\[E\varphi (z)=\sup\{ u(z):\,u\leq\varphi \text{ and } u  \text{ is plurisubharmonic on
} \D\}.\]

We define the class of measures $\widehat{\J}_z$ for $z\in \ovr\D$
in the following way: A measure $\mu \in C^*(\ovr \D)$ belongs to
$\widehat{\J}_z$ if there exist points $z_j\in \D$ and measures
$\mu_j \in \J_{z_j}$ such that $\{z_j\}$ converges to $z$ and
$\{\mu_j \}$ converges to $\mu$ in the weak-$*$ topology. A point $z
\in \partial \D $  is said to be \emph{$c$-regular} if the following
holds: If $z_j \in \D$, $z_j\to z $ and $\mu \in \widehat{\J}_{z}$,
then there exists a sequence of measures $\mu _{j}\in \J_{z_{j}}$
that converges weak-$*$ to $\mu.$ A domain is called $c$-regular if
every boundary point is $c$-regular. The following theorem
establishes the fact that regularity of plurisubharmonic envelopes
is equivalent to $c$-regularity of the domain (see
\cite{Gogus05,GogusThesis}).

\begin{theorem}[G\"{o}\u{g}\"{u}\c{s}] \label{T:reg1}
Let $\D$ be a bounded  domain in $\C^n$. All points in the boundary
of $\D$ are $c$-regular if and only if the envelope $E\varphi $ is
continuous on $\Dc$ for all functions $\varphi \in C(\ovr \D)$.
\end{theorem}

Plurisubharmonic upper envelopes of functions form an important tool
in pluripotential theory. One uses them in particular to construct
homogeneous solutions of the complex Monge-Amp\`{e}re operator with
continuous boundary data. In \cite{Gogus05} the first author used
Jensen measures to completely characterize the domains where
continuous functions have continuous  envelopes. For any function
$\psi\in C(\partial\D)$, the Perron-Bremermann function for $\psi$
is the function $S\psi$ defined on $\Dc$ by
\[S\psi (z)=\sup\{ u^*(z):\,u^*|_{\partial\D}\leq\psi \text{ and } u  \text{ is plurisubharmonic on
} \D\}.\] It is a well-known result due to Bremermann
\cite{Bremermann59}, Walsh  \cite{Walsh68}, Bedford and Taylor
\cite{BedfordTaylor76} that when $\D$ is a bounded strongly
pseudoconvex domain in $\C^n$, given a function $\psi\in
C(\partial\D)$, the function $S\psi$ is maximal plurisubharmonic in
the sense that $(dd^c S\psi)^n=0$ on $\D$, $S\psi$ is continuous,
and also $\lim_{w\to z}S\psi (w)=\psi (z)$ for every
$z\in\partial\D$. Thus $S\psi$ is the unique homogeneous solution of
the complex Monge-Amp\`{e}re operator with continuous boundary data
$\psi$ when the domain is strongly pseudoconvex. Related to this, it
was proved in \cite{Gogus05} that $c$-regularity of the domain is
equivalent to the continuity of Perron-Bremermann envelopes on
smooth bounded pseudoconvex domains in $\C^n$.

\begin{theorem}[G\"{o}\u{g}\"{u}\c{s}] \label{T:PerronBremermannCreg}
Let $\D$ be a smooth bounded pseudoconvex domain in $\C^n$. Then all
points in the boundary of $\D$ are $c$-regular if and only if the
Perron-Bremermann function $S\psi$ is continuous on $\Dc$ for every
$\psi\in C(\partial\D)$.
\end{theorem}

A point  $p\in \partial \D$ is said to be \emph{strongly regular} if
there exists a function $u\in PSH(\D)\cap C(\ovr \D)$ so that
$u(p)=0$ and $u(z)<0$ for every $z\in \ovr \D \backslash \{p\}$. The
motivation for this paper  comes from the following theorem of the
first author \cite{Gogus05}.

\begin{theorem}[G\"{o}\u{g}\"{u}\c{s}]\label{T:reg}
Let $\D$ be a bounded domain in $\C^n$. Assume that the boundary of
$\D$ is $c$-regular. Then the set  of strongly regular points is a
closed subset of the boundary of  $\D$.
\end{theorem}

In this paper we prove that Theorem \ref{T:reg} has a converse on
smooth bounded pseudoconvex domains in $\C^2$  and  that smooth
bounded pseudoconvex Reinhardt domains in $\C^{2}$ are $c$-regular
(see Theorem \ref{ThmClosedCregular} and Theorem \ref{ThmReinhardt}
in the next section).

The rest of the paper is organized as follows. The main results are
given in the next section. After that we give two examples  to show
that one cannot hope to have similar results as in Theorem
\ref{ThmClosedCregular} and Theorem \ref{ThmReinhardt} in $\C^n$ for
$n\geq 3.$ We postpone the proofs of the results until after the
examples. We conclude the paper with some remarks.

\section{Results}
Our first result states that Theorem \ref{T:reg} has a converse on
smooth bounded pseudoconvex domains in $\C^2.$

\begin{theorem} \label{ThmClosedCregular}
 Let $\D$ be a smooth bounded pseudoconvex domain in $\C^2.$ Then $\D$ is $c$-regular if
and only if the set of strongly regular points is closed in the
boundary of $\D.$
\end{theorem}

Combining Theorem \ref{T:reg1}, Theorem \ref{T:PerronBremermannCreg}
and Theorem \ref{ThmClosedCregular} we get the following corollary.

\begin{corollary}
Let $\D$ be a smooth bounded pseudoconvex domain in $\C^2.$ Then the
following statements are equivalent:
\begin{itemize}
\item[i)] The set of strongly regular points is closed in the boundary of $\D$.
\item[ii)] The envelope $E\varphi $ is continuous on $\Dc$ for
all functions $\varphi \in C(\ovr \D)$.
\item[iii)] The Perron-Bremermann function $S\psi$ is continuous on $\Dc$ for every
$\psi\in C(\partial\D)$.
\end{itemize}
\end{corollary}

We also show that smooth bounded Reinhardt domains in $\C^2$ are
$c$-regular. However, not all smooth bounded pseudoconvex domains in
$\C^2$ are $c$-regular. For example, a smooth bounded convex domain
whose boundary contains a single analytic disk is not $c$-regular.
See also Remark \ref{Rmk3} at the end of the paper for a discussion
about Hartogs domains in $\C^2.$

\begin{theorem}\label{ThmReinhardt}
Smooth bounded pseudoconvex Reinhardt domains in $\C^2$ are
$c$-regular.
\end{theorem}

Hence on smooth bounded pseudoconvex Reinhardt domains in $\C^2$ the
plurisubharmonic envelopes of functions continuous up to the
boundary are continuous up to the boundary.

\begin{corollary}
 Let $\D$ be a smooth bounded pseudoconvex Reinhardt domain in $\C^2.$ Then  the envelope
$E\varphi $ is continuous on $\Dc$ for all functions $\varphi \in
C(\ovr \D)$.
\end{corollary}

A  domain $\D$ is \emph{$B$-regular} if for any function $f\in
C(\partial \D)$ there exists a function $u\in PSH(\D)\cap C(\ovr
\D)$ so that its restriction $u|_{\partial \D}$ to $\partial \D$ is
$f.$ Sibony \cite{Sibony87,Sibony91} showed that  a hyperconvex
domain  is $B$-regular if and only if every boundary point is
strongly regular. The bidisk  is an example of a $c$-regular domain
that is not $B$-regular as the boundary contains analytic disks. It
turns out that,  for smooth bounded pseudoconvex $c$-regular
domains, this is the only obstruction for $B$-regularity. The
precise statement is as follows.

\begin{proposition}\label{P:Breg}
Let $\D$ be a smooth bounded pseudoconvex domain in $\C^n.$ Then the
following statements are equivalent:
\begin{itemize}
\item[i)] $\D$ is B-regular,
\item[ii)] $\partial \D$ is  strongly regular,
\item[iii)] $\partial \D$ is  c-regular and $\partial \D$ has no analytic disks.
\item[iv)]  The set of strongly regular points is a closed subset of $\partial
\D$ and $\partial \D$ has no analytic disks.
\end{itemize}
\end{proposition}

\section{Examples}
For simplicity we construct  examples in $\C^3.$ However, by a
simple modification one can obtain examples in $\C^n$ for $n\geq 3.$

Our first example shows that Theorem \ref{ThmClosedCregular} is not
true in $\C^3.$ We will construct a smooth bounded convex domain
$\D_1\subset \C^3$ that is not $c$-regular, yet the set of strongly
regular points of $\D_1$ is closed in $\partial \D_1.$ But first we
need a lemma.

Given a bounded domain $\D\sbs{\mathbb{C}}^n$ and a point $z\in \ovr
\D$, we denote by ${\J}^c_z=\J^c_z(\D)$ the family of all positive
Borel measures $\mu\in C^*(\ovr \D)$ such that $u(z)\leq \int
u\,d\mu$ for every $u$ in the set $PSH^c(\Dc)$ of all continuous
functions on $\ovr{\D}$ which are plurisubharmonic on  $\D$. In view
of \cite[Lemma 2.4]{Gogus05}, $\J_z\sbs\widehat{\J_z}\sbs\J^c_z$.

\begin{lemma}\label{Lem:Jc} Let $\D$ be a smooth bounded domain in ${\mathbb{C}}^n$,
$z\in\ovr{\D}$ be a point and $f:\ovr{\mathbb{D}}\to\Dc$ be a
mapping which is holomorphic in an open neighborhood of the closure
of the  unit disk $\ovr{\mathbb{D}}$ so that $f(0)=z$. Then
$\mu_f\in\J^c_z$.
\end{lemma}

\begin{proof} Let $u\in PSH^c(\Dc)$. By \cite[Theorem 1]{FornaessWiegerinck89} (Sibony
\cite[Th\'eor\`eme 2.2]{Sibony87} proved this result earlier in case
of pseudoconvex domains) there are smooth plurisubharmonic functions
$u_j$ on neighborhoods of $\ovr{\D}$ that approximate $u$ uniformly
on $\ovr{\D}$. Clearly $u_j(z)\leq \int u_j\,d\mu_{f}$ for every
$j$. Letting $j\to\infty$ we get  $u(z)\leq \int u\,d\mu_{f}$. Thus,
$\mu_f\in\J^c_z$.
\end{proof}

\begin{example}\label{ExampleClosedCregular}
We are going to construct a  bounded domain $\D_1\sbs\C^3$ with
smooth boundary such that
\begin{itemize}
\item[a.] the domain $\D_1$ is convex,
\item[b.] the set of strongly regular points of $\D_1$ is closed,
\item[c.] there exists a point in $\partial\D_1$ which is not $c$-regular.
\end{itemize}
Let $\psi$ be a non-negative smooth convex function on $[0,\infty)$
such that
\begin{itemize}
 \item[i.] $\psi=0$ on $[0,1]$
\item[ii.] $\psi'>0$ and $\psi''>0$ on $(1,\infty)$ and $\psi'(t)>1$
for $t\geq \sqrt{2}$
\item[iii.] $\psi(2)=3.$
\end{itemize}
Furthermore, let us define
\[r(z_1,z_2,z_3)=2\real(z_3)+\psi(|z_1|^2)+\psi(|z_2|^2)
+\psi(|z_3|^2)+|z_3|^2(12+a+|z_2|^2).\] By Sard's Theorem we can
choose $a>0$ sufficiently small that the domain
\[\D=\{(z_1,z_2,z_3)\in \C^3:r(z_1,z_2,z_3)<0\}\]
is smooth and bounded. Analysis of the eigenvalues of the Hessian of
the function $(x_3^2+y_3^2)(12+a+x_2^2+y_2^2)$  shows  that
$|z_3|^2(12+a+|z_2|^2)$ is convex for $|z_2|<2.$ Then $\D$ is a
smooth  bounded convex domain in $\C^3$ and the set
\[\Gamma=\left\{(z_1,z_2,z_3)\in\C^3: |z_1|\leq \frac{1}{2},|z_2|\leq
\frac{1}{2},z_3=0\right\}\subset \partial \D.\] Let $\Lev_r(z;X)$
denote the Levi form of $r$ at $z\in b\D$ applied to the complex
tangential vector $X.$ It is easy to see that there exist disks in
the boundary in $z_1$-direction through any point $(z_1,z_2,z_3)\in
\partial \D$ where $|z_1|<1$. This means that $\Lev_r(z;e_1)=0$ for
$e_1=(1,0,0)$ and $z=(z_1,z_2,z_3)\in \partial \D$ such that
$|z_1|<1$ and $z_3\neq 0.$  Let
\[W=\left(0,1+\zb_3(12+a+|z_2|^2),-\zb_2|z_3|^2\right).\]
Then  $W$ is a complex tangential direction at $z\in \partial \D$
near $\Gamma$ that is perpendicular to $e_1.$ One can calculate that
\begin{align*}
\Lev_r(z,W)=&|z_3|^2\left(
\left|1+z_3(12+a+|z_2|^2)\right|^2-2\real\left(z_2^2\zb_3(1+\zb_3(12+a+|z_2|^2))
\right)\right)\\
&+|z_3|^2 \left(|z_2|^2|z_3|^2(12+a+|z_ 2|^2) \right) \\
=&|z_3|^2(1+O(|z_3|)) \text{  for } z \text{ near }\Gamma.
\end{align*}
Then there exists a neighborhood $U$ of $\Gamma$ such that
$\Lev_r(z,W)>0$ for $z\in U\setminus \Gamma_1$ where
$\Gamma_1=\{(z_1,z_2,z_3)\in \C^3:z_3=0\}.$ Therefore, we showed
that there are only analytic disks, one dimensional complex
manifolds, in $\partial \D$ through  $(z_1,z_2,z_3)\in (U\cap
\partial \D) \setminus \Gamma_1$ such that $|z_1|<1.$

One can check that $r$ is strongly plurisubharmonic at
$(z_1,z_2,z_3)$ if $|z_1|>1$ and $z_3\not =0.$ The fact that there
is a disk (given by $\xi \to (e^{i\theta},\xi,0)\subset \Gamma_1$
for $|\xi|<1$ and $\theta\in \re$) in $\partial\D$ in
$z_2$-direction through $(e^{i\theta},0,0)$ implies that set of
strongly regular points in the boundary of $\D$ is not closed (one
can approximate $(1,0,0)$ by a sequence $\{\eta_j\}\subset \partial
\D$ where the norm of the first component of $\eta_j$ is strictly
larger than 1 and the last component is nonzero). So we need to
modify the domain further to make sure that strongly regular points
in the boundary are  closed. Using Sard's Theorem, again if
necessary, we can choose a small  $\delta>0$ and a smooth, convex,
and non-decreasing function $\lambda:\mathbb{R}\to[0,\infty)$ such
that
\begin{itemize}
 \item[i.] $\lambda(t)=0$ for $|t|\leq \delta$ and $\min\{ \lambda'(t),\lambda''(t)\}>0$ for
$t>\delta,$
\item[ii.]  $\min\{\lambda(|z_1|^2+|z_2|^2+|z_3|^2):(z_1,z_2,z_3)\in
\C^3\setminus U \}>-\min\{r(z): z\in \D\}.$
\item[iii.] $\D_1=\{(z_1,z_2,z_3)\in \C^3:r(z_1,z_2,z_3)+
\lambda(|z_1|^2+|z_2|^2+|z_3|^2)<0\}$ is a smooth bounded convex
domain in $\C^3.$
\end{itemize}
We note that $\D_1 \subset U\cap \D$ and  the set of weakly
pseudoconvex points of $\D_1$ is
\[\partial \D_1\cap \{(z_1,z_2,z_3)\in \C^3:|z_1|^2+|z_2|^2+|z_3|^2\leq \delta\}.\]
A point in the boundary of a bounded convex domain in $\C^n$ is
strongly regular if and only if there are no analytic disks in the
boundary through the point (see \cite[Proposition
3.2]{FuStraube98}). Since there is no analytic disk in
$\{(z_1,z_2,z_3)\in \C^3:|z_1|^2+|z_2|^2 +|z_3|^2= \delta\}$ we
conclude that the set of strongly regular points of $\D_1$ is
\[\partial \D_1 \setminus \{(z_1,z_2,z_3)\in \C^3:|z_1|^2+|z_2|^2+|z_3|^2< \delta\}.\]
Hence it is closed. However, $\D_1$ is not $c$-regular. To see this,
first notice that there exists a disk in the boundary of $\D_1$ with
center $p=(0,0,0)$ in the $z_2$-direction. For example, we can take
the analytic disk defined by the map $f(\zeta)=(0,\delta \zeta, 0).$
By Lemma \ref{Lem:Jc} the measure $\mu_f\in\J^c_z$. Let
$p_j=(0,0,w_j)$ where
\[w_j=-\frac{1}{j}+i\sqrt{\frac{2j-12-a}{j^2(12+a)}}.\]
One can check that $\{p_j\}\subset \partial \D_1$ is a sequence
converging to $p$ and,  by our construction, any  analytic disk
through $p_j$ is in the $z_1$-direction.
\smallskip

\noindent \textit{Claim: If $\mu_j\in \J_{p_j}$ then   $\mu_j$ is
supported on the  disk through $p_j$ in $z_1$-direction.}
\smallskip

\noindent \textit{Proof of Claim:} This can be seen as follows:
$\D_1$ is convex and the transversal direction to each disk through
$p_j$ is strongly pseudoconvex. Then there exist linear functionals
$L_j$ such that $L_j< 0$ on $\Dc_1\backslash \{(z_1,0,w_j)\in
\C^3:z_1\in\mathbb{C}\}$ and $L_j=0$ on
$\{(z_1,0,w_j)\in\C^3:z_1\in\mathbb{C}\}$. Hence every Jensen
measure of $p_j$ is supported on the disk through $p_j$. This
finishes the proof of the claim.
\smallskip

Let $\varphi \in C(\overline{\D}_1)$ be a function so that $\varphi$
depends on $z_2$ only, $-1\leq\varphi\leq 0$, $\varphi (p)=0$, and
$\varphi (z_1,z_2,z_3)=-1$  when $|z_2|=\delta$. Hence  $\mu
_f(\varphi)=-1$. In view of the observation above
$\lim_{j\to\infty}\mu_j(\varphi)=\lim_{j\to\infty}\varphi(p_j)= 0$.
Thus $\mu_j$  does not converge to $\mu_f.$ From \cite[Corollary
4.4]{Gogus05} it follows that $p$ is not $c$-regular.
\end{example}

Next we would like to give an example that shows that Theorem
\ref{ThmReinhardt} is not true in $\C^3.$ First we need the
following proposition whose proof is essentially contained in
\cite{HerbigMcNeal} in the proof of Proposition 6.17. We will
therefore skip the proof.

\begin{proposition}\label{PropConvex}
Let $\D$ be a smooth bounded convex  domain in $\mathbb{R}^n.$ Then
$\D$ has a defining function that is strictly convex  on $\D.$
\end{proposition}

\begin{remark}
Observation of the proof in \cite{HerbigMcNeal} also shows that a
similar statement is true for pseudoconvex domains with
plurisubharmonic defining functions. That is, if a smooth bounded
pseudoconvex domain $\D$ in $\C^n$ has a plurisubharmonic defining
function then it also has a defining function that is strictly
plurisubharmonic on $\D.$
\end{remark}

\begin{example}\label{ExampleReinhardt}
In this example we will construct a smooth bounded complete
Reinhardt pseudoconvex
 domain in $\C^3$ that is not $c$-regular. Hence the statement of Theorem \ref{ThmReinhardt} is not true
in $\C^3.$

Let $H_1=\{(x,y)\in \mathbb{R}^2: -1< x,y< 1\}$ and $H_2$ be a
smooth bounded convex domain that is obtained  by smoothing out the
corners of $H_1$ so that $\{(x,y)\in \mathbb{R}^2: x=1, 0\leq y\leq
1/2\} \subset \partial H_2.$ Since $H_2$ is convex it has a convex
defining function. Then by Proposition \ref{PropConvex} we can
choose a defining function $\rho$ for  $H_2$ that is strictly convex
on $H_2.$  Let us define
 \[\D=\left\{(z_1,z_2,z_3)\in \C^3: |z_3|^2+\rho(|z_1|^2,|z_2|^2)<0\right\}.\]
Then $\D$ is a smooth bounded complete Reinhardt domain in $\C^3.$
One can show that $\D$ is pseudoconvex and has a disk
\[\Delta =\{(1,\zeta/2 ,0):\zeta\in\C,\,|\zeta|<1\}\]
in the boundary centered at $p=(1,0,0)$ yet all the points on the
boundary which belong to the set $R=\{(z_1,z_2,z_3)\in \C^3:z_3\neq
0\}$ are strongly pseudoconvex. Then the set of strongly regular
points contains the set $\partial \D\cap R$ (strongly pseudoconvex
points are strongly regular) yet $p$ is not strongly regular. Hence
the set of strongly regular points in $\partial \D$ is not closed in
$\partial \D.$ Therefore, Theorem \ref{T:reg} implies that $\D$ is
not $c$-regular.
 \end{example}

\section{Proofs}
\begin{proof}[Proof of Proposition \ref{P:Breg}]
i) and ii) are equivalent by Sibony's results
\cite{Sibony87,Sibony91}. It is clear that ii) implies iii) and iii)
implies iv) by Theorem \ref{T:reg}. So we only need to prove that
iv) implies ii). Assume that the set of strongly regular points is
closed,  $\partial \D$ has no analytic disks, and there exists $p\in
\partial\D$ that is not strongly regular. Then the Levi form of $\D$ degenerates
at $p$. Since the set of strongly regular points  is closed, and not
the whole of $\partial \D$, the Levi form degenerates on an open set
$U$ (in the relative topology) in the boundary of $\D$. Then there
exists a relatively open set $V\subset U$ in the boundary where the
Levi form is of constant rank. This implies that $V$ is locally
foliated by complex manifolds (see, for example, \cite{Freeman74}).
We reach a contradiction. Therefore, $\partial \D$ is strongly
regular.
\end{proof}

\par The following observation will be needed in the proof of Theorem
\ref{ThmClosedCregular}.

\begin{lemma}\label{Lem:JensenMeasureDisk}
Let $\Delta$ be the unit disk in $\C$, $\D\sbs\mathbb{C}^2$ be a
bounded domain, and $W$ be a regular  bounded domain in $\C$ such
that $0\in \partial W, \Dc \subset \Delta \times W,$ and $p=(0,0)\in
\partial \D$.
 If $\mu\in \widehat{\J}_p$, then $\supp \mu\subset \Delta \times \{0\}.$
\end{lemma}

\begin{proof}
Since $W$ is regular there exists a subharmonic function $u$  such
that $u(0)=0$ and $u<0$ on $W\setminus \{0\}.$ Let $U(z_1,z_2)=
u(z_2).$ Then $U$ is a plurisubharmonic function on $\Delta \times
W$ such that $U(z,0)=0$ for $z\in \Delta$ and $U<0$ on
$\D\setminus(\Delta\times\{0\}). $ This implies that $\supp
\mu\subset \Delta \times \{0\}.$
\end{proof}

We recall from \cite{Gogus08} that a bounded domain $\D\sbs \C^n$ is
locally $c$-regular if for every point $z\in\partial\D$ there is an
open neighborhood $N$ of $z$ so that $V=\D\cap N$ is $c$-regular.
The localization result in \cite{Gogus08} says that a bounded domain
$\D$ is $c$-regular if and only if it is locally $c$-regular.

\begin{proof}[Proof of Theorem \ref{ThmClosedCregular}]
One direction follows from Theorem \ref{T:reg}. For the other
direction assume that strongly regular points are closed. If all
points in the boundary are strongly regular, then the boundary is
$B$-regular. This clearly implies that $\D$ is $c$-regular.

Suppose that the set of strongly regular points does not cover the
whole boundary. Let $p$ be a boundary point  that is not strongly
regular. By the localization result in \cite{Gogus08} it is enough
to prove that there exists an open ball $U$ containing $p$ such that
$p$ is $c$-regular in $\D_1=\D\cap U.$  Since $\D\subset \C^2$ and
the set of strongly regular points is closed, the  Levi form is of
constant rank in a neighborhood of $p$ in the boundary. Then we can
choose $U$ so that $U\cap \partial \D$ is foliated by complex disks.

Using a holomorphic change of coordinates  we may assume that $p$ is
the origin, the disk $\Delta$ containing $p$ sits in the first
coordinate, and positive $y_2$-axis is the outward normal to the
boundary  on $\Delta.$ Then there exist $r,\theta >0$ so that
$\Delta\times \{0\}\subset \Dc_1 \subset \mathbb{C}\times \ovr W$
where $W=\{z\in \C:|z|<r, |\text{Arg}(z)+\pi/2|<\theta \}.$ Since
$W$ is regular Lemma \ref{Lem:JensenMeasureDisk} applies and
 all Jensen measures of $p$ on $\D_1$ are supported on $\Delta\times \{0\}.$ Let $\mu\in
\widehat \J_p$ be a measure in $\D_1$ and $\{q_j\}$ be a sequence in
$\D_1$ that converges to $p.$ Since an open neighborhood (in
$\partial \D_1$) of $p$ is foliated by analytic disks the foliation
can be parametrized. That is, there exists a smooth mapping
$F:\Delta \times (-1,1)\to \partial\D_1$ such that $F$ is a
diffeomorphism  onto its image, $F(z,0)=z$ on $\Delta,$ and $F$ is
holomorphic in $z$ (see, for example, Theorem 1.1 and Theorem 2.14
in \cite{Freeman74}). Let $\eta$ be the outward normal vector
$(0,i)$ to the boundary of $\D _1$ on the disk $\Delta$. Then there
exists a number $a>0$ so that the function $G(z,s,t)=F(z,s)+t\eta$
maps $H=\Delta\times (-a,a)\times (-a,0]$ into $\ovr \D_1$. For any
$\alpha\in\Delta$ let
\[G_{\alpha,\beta}(z,s,t)=G\left(\frac{\beta z+\alpha}{1+\ovr\alpha \beta z},s,t\right)
\text{ for } (z,s,t)\in H.\] There exist points
$(\alpha_j,s_j,t_j)\in H$ such that $t_j<0$  and
$G(\alpha_j,s_j,t_j)=q_j$ for all $j,$ and $s_j\to 0$ and $t_j\to 0$
as $j\to \infty.$ We consider $\mu$ as a measure on $\Delta\times
\{0\}\times\{0\}\sbs H$. Let us choose numbers $0<\beta_j<1$ such
that $\beta_j\to 1$ and the functions
\[g_j(z)=G_{\alpha_j,\beta_j}(z,s_j,t_j),\,\,\,\,z\in\ovr\Delta\] map $H$ into $\D_1$.
Let $\mu_{j}=(g_j)_* \mu$ (here $(g_j)_*$ denotes the push forward
by $g_j$). Since  the sequence of functions $\{G(z,s_j,t_j)\}$
converges to $G(z,0,0)$ uniformly on $\overline\Delta$, the sequence
of measures $\{\mu_j\}$ converges weak-$*$ to $\mu$ as $j\to
\infty$. Notice that the measures $\mu_{j}$ are supported in $\D_1$.
Moreover for every $j$ we have
\[\int \varphi\,d\mu_j=\int\varphi \circ g_j\,d\mu \]
 for every $\varphi\in C(\ovr\D_1)$.  If $u\in PSH^b(\D)$, then
\[u\circ g_j(0)=u(q_j)\leq \int u \circ g_j\,d\mu =\int u^*\,d\mu_j\]
for every $j\geq 1$. Hence the measure $\mu_j\in\J_{q_j}$. This
completes the proof of Theorem \ref{ThmClosedCregular}.
\end{proof}

\begin{proof}[Proof of Theorem \ref{ThmReinhardt}]
 Let us assume that $\D\sbs\C^2$ is a smooth bounded pseudoconvex  Reinhardt domain,
$\theta^1_{zw}(t)=(e^{it}z,w),$ and $\theta^2_{zw}(t)=(z,e^{it}w).$
Then $\frac{d}{dt}\theta^{1}_{zw}(t)=(ie^{it}z,0) $ and
$\frac{d}{dt}\theta^2_{zw}(t)=(0,ie^{it}w).$  Assume that $p\in
\partial \D$ and there exists an analytic disk $\Delta$ in the
boundary of $\D$ through the point $p.$

First suppose that $p$ is away from coordinate axes. Replacing
$\Delta$ by a smaller disk if necessary we may assume that $\Delta$
is away from the coordinate axes as well. Then
$\frac{d}{dt}\theta^{1}_{zw}$ and $\frac{d}{dt}\theta^{2}_{zw}$ span
$\C^2$ on $\Delta.$ The fact that $\Delta$ is two real dimensional
and that the rotations around $z$ and $w$ axes add at least one more
dimension imply that there is an open set in the boundary around $p$
that is foliated by analytic disks. We recall that pseudoconvex
Reinhardt domains are convexifiable away from the coordinate axes
and  strongly regular points of a convex domain are precisely the
ones with no disks through them  in the boundary. Then the set of
points in $\partial \D$ (away from the coordinate axes) that are not
strongly regular is a relatively open set.

Now let us assume that $p$ is a boundary point on the coordinate
axes. By the open mapping property of holomorphic functions in one
variable there are no nontrivial analytic disks in the intersection
of the  boundary with the coordinate axes. Hence if there is a disk
through $p$ then we can rotate the disk as in the previous case to
show that there are disks in a neighborhood of $p.$ So without loss
of generality assume that $p=(q,0)$ is in the boundary of $\D$ and
there is no analytic disk in  the boundary of  $\D$ passing through
$(q,0).$ Then there is a smooth subharmonic  function $\phi$ on the
disk $\Delta(|q|)\subset \C$ centered at the origin and radius $|q|$
such that $\phi(q)=0$ and $\phi<0$ on
$\overline{\Delta(|q|)}\setminus \{q\}$. Let $\psi(z,w)=\phi(z).$
Then $\psi(z,w)<0$ when $|z|<|q|.$ However, since there is no
analytic disk through $p$ and $\D$ is smooth Reinhardt the condition
$(z,w)\in \Dc\setminus \{p\}$ implies that  $|z|<|q|.$ Hence, the
function $\psi$ peaks at the point $p.$

Therefore, we showed that  the set of strongly regular points in
$\partial \D$ is closed and Theorem \ref{ThmClosedCregular} implies
that $\D$ is $c$-regular.
\end{proof}

\section{Further Remarks}
\begin{enumerate}
\item We note that pseudoconvexity is necessary in Proposition \ref{P:Breg}.  Let
$B(p,r)$ denote the ball centered at $p$ with radius $r.$ Now let
$p_1=(0,0)$ and $p_2=(2,0)$. By smoothing the domain
$B(p_1,2)\setminus B(p_2,1)\subset \C^2$ we can obtain a smooth
bounded $c$-regular domain $\D$ such that  for every $z\in \partial
B(p_2,1)\cap B(p_1,2)\cap \partial \D$ there are disks in $\Dc$
passing through $z.$ Hence $\D$ is not $B$-regular and there are no
analytic disks in the boundary.

\item Whether  the cross product of $c$-regular domains is $c$-regular is still an
open problem. However, the cross product of two $B$-regular domains
is $c$-regular. That is, if $U\sbs\mathbb{C}^n$ and
$V\sbs\mathbb{C}^m$ are bounded $B$-regular domains then $U\times V$
is $c$-regular. To see this let $p=(z_0,w_0)\in\partial (U\times
V)$. If $z_0\in\partial U$ and
 $w_0\in\partial V$, then there are functions
$u\in PSH(U)\cap C(\ovr U), v\in PSH(V)\cap C(\ovr V)$ so that
$u(z_0)=0$, $v(w_0)=0$, $u(z)<0$ for $z\in \ovr U\setminus \{z_0\}$,
$v(w)<0$ for $w\in \ovr V\setminus \{w_0\}$. Then clearly the
function $\rho (z,w)=u(z)+v(w)$ is a plurisubharmonic peak function
on $U\times V$ for $(z_0,w_0)$ and the point $(z_0,w_0)$ is strongly
regular.  So let us assume without loss of generality that $z_0\in
\partial U$ and $w_0\in V$. Let $\mu\in\J^c_p(U\times V)$. We claim
that the support of $\mu$ belongs to the set $S=\{(z_0,w):w\in V\}$.
Suppose this is not the case. This means that there is a compact set
$K\sbs U$ so that $z_0\not\in K$ and $\mu (K\times \ovr V)>c$ for
some $c>0$. Let $u\leq 0$ be a plurisubharmonic peak function on $U$
for $z_0$ as above. Let $r(z,w)=u(z)$ for $z\in\ovr U$, $w\in\ovr
V$. Then for any $w\in\ovr V$ we have
\[0=r(z_0,w)=u(z_0)\leq \int _Kr\,d\mu\leq c\max _Kr<0.\]
The contradiction proves the claim that $\supp \mu\sbs S$. Thus
$\mu=\delta_{z_0}\times \nu$ for some probability measure $\nu$ on
$\ovr V$. One can show that $\nu\in\J^c_{w_0}(V)$. In fact, if $v\in
PSH^c(V)$, then let us consider the function $v_0(z,w)=v(w)$ defined
on the product $U\times V$. One sees that
\[v(w_0)=v_0(z_0,w_0)\leq \int v_0\,d\mu=\int v\,d\nu.\]
By \cite[Corollary 4.4]{Gogus05} we have $\J^c_z(V)=\J_z(V)$, hence,
we see that the measure  $\mu\in\J_p(U\times V)$. Thus $\J_p(U\times
V)=\J^c_p(U\times V)$, and again by \cite[Corollary 4.4]{Gogus05}
the point $p$ is $c$-regular.

\item Let $U\sbs{\mathbb{C}}^n$ and $V\sbs{\mathbb{C}}^m$ be domains.
Then the set of strongly regular points in $\partial(U\times V)$ is
closed if and only if the strongly regular points in $\partial U$
and $\partial V$ are closed. Suppose the strongly regular points in
$\partial U$ and $\partial V$ are closed. If $(p_j,q_j)\in\partial
(U\times V)$ are strongly regular and converging to $(p,q)$, then
clearly $p$ and $q$ are strongly regular points in $\partial U$ and
$\partial V$ respectively. Let $u(z)$ and $v(w)$ be the
plurisubharmonic peak functions in $U$ and $V$ for $p$ and $q$
respectively. The function $s(z,w)=u(z)+v(w)$ is clearly a
plurisubharmonic function on $U\times V$ so that $s(p,q)=0$,
$s(z,w)<0$ when $(z,w)\not =(p,q)$. Thus,  the strongly regular
points in $\partial(U\times V)$ are closed.

To prove the other direction assume that the strongly regular points
in $\partial (U\times V)$ are closed and without loss of generality
assume that $\{p_j\}$ is a sequence of strongly regular points in
$\partial U$ that converges to a point $p$. If $q$ is a strongly
regular point in the boundary of $V$, then $(p_j,q)$ are strongly
regular points in $\partial (U\times V)$ that converges to $(p,q)$.
Let $\varphi \leq 0$ be a plurisubharmonic function on $U\times V$
so that $\varphi (p,q)=0$, $\varphi (p',q')<0$ when $(p',q')\not
=(p,q)$. Then $u(z)=\varphi (z,q)$ is a plurisubharmonic function on
$U$ that peaks at $p$. Hence the set of strongly regular points in
$\partial U$ is closed. In a similar way one can show this for
$\partial V$.

\item \label{Rmk3} P. Matheos in his thesis \cite{MatheosThesis}  constructed a smooth
bounded pseudoconvex complete Hartogs domain in $\C^2$  that has no
analytic disk in the boundary and yet not $B$-regular. Hence  by
Proposition \ref{P:Breg} this domain is not $c$-regular and there is
no analog of Theorem \ref{ThmReinhardt} for Hartogs domains.
\end{enumerate}

\section*{acknowledgement}
We would like to thank Evgeny Poletsky for valuable comments on a preliminary version of
this manuscript.

\end{document}